\documentclass[12pt]{amsart}
\usepackage[T1]{fontenc}
\usepackage[utf8]{inputenc}
\usepackage{amsmath,amssymb,amsthm,mathtools,bm}
\usepackage{geometry}
\usepackage[hypertexnames=false]{hyperref}
\hypersetup{colorlinks=true, linkcolor=blue, citecolor=blue, urlcolor=blue}
\usepackage{enumitem}

\newtheorem{theorem}{Theorem}[section]
\newtheorem{proposition}[theorem]{Proposition}

\theoremstyle{definition}
\newtheorem{definition}[theorem]{Definition}

\theoremstyle{remark}
\newtheorem{remark}[theorem]{Remark}

\newcommand{\oconn}{\mathring{\Gamma}}

\newcommand{\tLambda}{\widetilde{\Lambda}}

\begin{document}

\title{The Geometry of Dilation- and Shear-Deformed Spaces}
\author{Gordon Liu}
\address{Independent Scholar, Melbourne, Australia}
\email{gordonliu168@gmail.com}

\begin{abstract}
This paper develops a deformation-field geometry for spaces whose local frames may undergo internal dilation, compression, and shear.  The basic datum is an admissible dilation--shear field \(P\) over a selected metric-compatible reference geometry \((M,\bar g,\bar\nabla)\).  It represents the induced metric by \(g=P^T\bar gP\) and compares tangent data through the reference representative \(\bar V=PV\).  The covariant derivative associated with the natural connection is defined by
\[
 \nabla_XV=P^{-1}\bar\nabla_X(PV),
\]
with local connection coefficients
\[
 \Gamma=\Lambda=P^{-1}\bar\Gamma P+P^{-1}dP .
\]
Thus the total dilation--shear compensation is represented by the natural connection coefficients \(\Gamma=\Lambda\).  If \(\bar\nabla\bar g=0\), then \(\nabla g=0\); hence the covariant derivative associated with the natural connection has zero nonmetricity.  The general distinction from Levi-Civita geometry lies instead in torsion and in the deformation origin of the comparison.  The Levi-Civita connection coefficients \(\mathring\Gamma[g]\) are retained as the torsion-free metric connection of the induced metric layer; they appear naturally in fully isometric realizations, but they are not an additional term to be added to the natural connection.  The same pullback rule covers composite references that already contain both an isometric realization and an internal dilation--shear natural connection.  Examples involving one-dimensional dilation, conformal deformation, anisotropic dilation, shear, and spherical geometries distinguish metric representation, natural dilation--shear comparison, Levi-Civita comparison, and embedded realization.
\end{abstract}


\maketitle

\section{Introduction}

Differential geometry provides a powerful language for studying smooth metric manifolds, their Levi-Civita connections, and their curvature.  In the standard Riemannian viewpoint, one starts with a metric space \((M,g)\), and the Levi-Civita connection is then singled out by metric compatibility and torsion-freeness \cite[pp.~124--126]{Lee}\cite[pp.~54--55]{doCarmo}.  This gives a canonical intrinsic comparison law: parallel transport preserves lengths and inner products with respect to the given metric.

The point of departure of the present paper is that a geometric body may carry another kind of structure in addition to smooth bending or curvature.  Besides being curved as a metric object, a space may have internal dilation, compression, or shear with respect to a chosen reference geometry.  This is the geometric situation that motivated the present theory: most standard geometric descriptions emphasize smooth curved shapes, while the possibility of intrinsic non-uniform dilation--shear is usually not separated as an independent comparison structure.  Conversely, a curved geometric body can often be understood as one way in which dilation--shear information is represented geometrically, for example after an isometric realization in an ambient space.  The goal here is to formulate a geometry that keeps track of both the ordinary Riemannian metric comparison and the internal dilation--shear comparison.

A simple curve already illustrates the motivation.  Classical curve geometry describes how a curve bends in a plane, how it twists out of the plane, and how its moving frame rotates in an ambient space.  One may also ask, however, whether the curve itself carries non-uniform dilation or compression along its tangent direction, or more generally whether neighboring curve elements possess an internal dilation--shear comparison.  Such information is not described by the curvature and torsion of the embedded curve alone.  The present theory is designed to retain this internal deformation structure and to extend the same idea from curves to surfaces and higher-dimensional manifolds.

The basic object is a deformation field \(P\).  It is modeled on the non-rotational, dilation--shear sector of \(GL(n,\mathbb R)\) with respect to \(\bar g\): pointwise \(\bar g\)-orthogonal rotations, and boosts in pseudo-Riemannian signature, are excluded because they preserve \(\bar g\) and therefore do not generate a new metric in \(g=P^T\bar gP\).  The field \(P\) has two related roles.  First, with respect to a fixed reference metric \(\bar g\), it deforms the metric frame and represents the metric
\begin{equation}\label{eq:intro_metric}
 g=P^{T}\bar g P .
\end{equation}
Second, \(P\) is used as a comparison map between the deformed geometry and the reference representation: a tangent vector \(V\) in the deformed space has reference representative \(\bar V=PV\), while \(P^{-1}\) returns reference representatives to the deformed space.  Thus \(P\) is not a coordinate Jacobian and is not merely a change of notation for the metric; it both shapes the metric frame and provides the push-back/push-forward comparison used in the first-order derivation of the dilation--shear compensation.

The Levi-Civita connection of the represented metric is denoted by \(\mathring{\Gamma}[g]\).  It is the canonical metric-preserving and torsion-free comparison associated with the metric layer \((M,g)\).  In the present deformation viewpoint, however, this connection is derived from the metric; it is not the natural comparison rule of the dilation--shear deformation field itself.  This distinction is especially clear in a fully isometric realization: the Levi-Civita connection is the tangential metric-preserving connection of the realized Riemannian geometry, while the dilation--shear data are determined by the deformation field and the selected reference comparison.

The dilation--shear comparison is obtained by comparing reference representatives.  For a deformed-space vector \(V\), the reference representative is \(\bar V=PV\).  The natural-parallel condition is that \(PV\) be parallel in the selected reference geometry.  This gives
\begin{equation}\label{eq:intro_natural_connection}
 \nabla_XV:=P^{-1}\bar\nabla_X(PV).
\end{equation}
If the reference covariant derivative \(\bar\nabla\) has local connection coefficients \(\bar\Gamma\), then the covariant derivative \(\nabla\) has local connection coefficients
\begin{equation}\label{eq:intro_connection_coefficients}
 \Gamma=P^{-1}\bar\Gamma P+P^{-1}dP.
\end{equation}
The total dilation--shear compensation is denoted by \(\Lambda\), and in this natural comparison sense
\begin{equation}\label{eq:intro_Lambda_equals_Gamma}
 \Lambda=\Gamma.
\end{equation}
Only in a fixed flat Euclidean or Minkowskian Cartesian baseline, where \(\bar\Gamma=0\), does this reduce to
\[
 \Gamma=\Lambda=P^{-1}dP .
\]
For an intrinsic final space one uses the total deformation field and the selected fixed baseline to compute the corresponding total coefficients.  More generally, the selected reference may already combine an isometric realization with an internal dilation--shear natural connection.  Such a composite reference is still represented only by its current metric-compatible comparison \((\bar g,\bar\nabla,\bar\Gamma)\).  A further deformation is then governed by the same formula \eqref{eq:intro_connection_coefficients}, rather than by a separate embedding correction or a Levi-Civita summand.

A central consequence is metric compatibility.  Whenever the reference comparison is metric-compatible, \(\bar\nabla\bar g=0\), the covariant derivative \(\nabla\) defined by \eqref{eq:intro_natural_connection} is metric-compatible with the induced metric \(g=P^T\bar gP\).  Thus the deformation spaces considered here have no nonmetricity with respect to the covariant derivative associated with the natural total connection.  Their difference from the Levi-Civita metric layer is generally expressed by torsion and by the deformation origin of \(\Gamma\), not by nonmetricity.

Consequently, the expression \(\mathring\Gamma+\Lambda\) is not used as the natural total connection of the deformation geometry.  The natural total connection is represented by \(\Gamma=\Lambda\).  If one later compares these coefficients with the Levi-Civita connection coefficients of the same induced metric, the difference
\[
 D:=\Gamma-\mathring\Gamma[g]
\]
is only a relative distortion tensor between two already specified connections.  It should not be confused with the dilation--shear total connection \(\Lambda\), and it should not be used to redefine the natural comparison.

If no deformation-comparison data are retained, the framework reduces to the ordinary Riemannian metric-only case
\begin{equation}\label{eq:intro_riemannian_limit}
 \nabla=\mathring\nabla[g], \qquad \Gamma=\mathring\Gamma[g].
\end{equation}
Thus Riemannian geometry keeps only the metric, its curvature, and the Levi-Civita comparison, while the present framework also retains the dilation--shear comparison whenever such deformation-comparison data are part of the geometry.

Embedded realizations are used as examples and comparisons: they show how an intrinsic metric geometry may be represented by ambient bending, without defining the dilation--shear compensation itself.

The paper is organized as follows.  Section~\ref{sec:metric} introduces deformation fields and the induced metric.  Section~\ref{sec:connections} constructs the natural connection, proves the uniform rule for composed and composite references, and clarifies its relation to the Levi-Civita connection of the induced metric layer.  Section~\ref{sec:torsion_curvature} records torsion, nonmetricity, curvature, the flat-baseline Maurer--Cartan identity, and the relation with Levi-Civita curvature.  Section~\ref{sec:examples} gives examples.  Section~\ref{sec:relations} briefly locates the framework with respect to Riemannian geometry, embedding geometry, metric-affine geometry, symmetric teleparallel geometry, and conformal geometry.

\section{Deformation Fields and Deformed Metrics}\label{sec:metric}

\subsection{Reference geometries and deformation fields}

Let $M$ be a smooth $n$-dimensional manifold.  A deformation field can be introduced with respect to any selected reference metric geometry
\begin{equation}
 (M,\bar g),
\end{equation}
which may be flat or curved, intrinsic, or already realized as an embedded Riemannian geometry.  Thus the deformation field is not tied to a Euclidean or Minkowskian reference for the purpose of metric representation.  Once a reference metric \(\bar g\) and an admissible deformation field \(P\) are selected, the represented metric of the new space is
\begin{equation}
 g=P^T\bar gP .
\end{equation}
In this sense the theory studies dilation--shear deformations of arbitrary reference spaces.

\subsection{Reference and deformed coordinates}\label{subsec:reference_deformed_coordinates}

We distinguish the coordinate notation of the reference representation from that of the deformed space. Reference coordinates, when used, are denoted by
\[
 \bar x^A,\qquad \bar\partial_A=\frac{\partial}{\partial \bar x^A},\qquad d\bar x^A,
\]
where capital Latin indices \(A,B,\ldots\) refer to fixed-reference components. Coordinates on the deformed space are denoted by
\[
 x^\mu,\qquad \partial_\mu=\frac{\partial}{\partial x^\mu},\qquad dx^\mu,
\]
where Greek indices \(\mu,\nu,\ldots\) refer to deformed-space components.

The same field \(P\) enters in two distinct ways. In the metric construction
\begin{equation}\label{eq:coordinate_distinction_metric}
 ds^2=g_{\mu\nu}(x)dx^\mu dx^\nu
 =\bar g_{AB}\bar\theta^A\bar\theta^B,
 \qquad
 g_{\mu\nu}=\bar g_{AB}P^A{}_{\mu}P^B{}_{\nu},
\end{equation}
it acts as the dilation--shear deformation field, with
\begin{equation}\label{eq:reference_coframe_representative}
 \bar\theta^A:=P^A{}_{\mu}dx^\mu .
\end{equation}
In the comparison of tangent data, it gives the fixed-reference representative of a deformed-space vector:
\begin{equation}\label{eq:coordinate_distinction_barV}
 \bar V^A=P^A{}_{\mu}V^\mu .
\end{equation}
The coframe \(\bar\theta^A\) is therefore a fixed-reference representative of the deformed displacement, not necessarily an exact differential \(d\bar x^A\). Only in the special integrable case in which \(P\) is the Jacobian of a coordinate map can one write \(\bar\theta^A=d\bar x^A\). Thus barred indices and Greek indices should not be identified unless an explicit frame identification has been chosen, and \(P\) is not, in general, a coordinate Jacobian.

\subsection{Reference representation principle}\label{subsec:reference_representation_principle}

The metric law used in this paper is a pointwise reference-representation principle.  The deformation field maps a tangent vector in the deformed description to its fixed-reference representative,
\begin{equation}\label{eq:reference_representation_principle_barV}
 \bar V=P V .
\end{equation}
The representative describes the same geometric vector in the selected reference metric, and hence inner products are represented by
\begin{equation}\label{eq:reference_representation_principle_inner_product}
 g(V,W)=\bar g(PV,PW).
\end{equation}
Equivalently,
\begin{equation}\label{eq:reference_representation_principle_metric}
 g=P^T\bar g P.
\end{equation}
Thus the same field \(P\) supplies both the metric representation and the map used to compare tangent data with the reference.  This principle is pointwise and frame-theoretic; it is not, in general, a coordinate transformation.

\begin{definition}[Reference representation principle]\label{def:reference_representation_principle}
A deformation field \(P\) represents tangent vectors of the deformed geometry in the fixed reference geometry by \(\bar V=PV\).  The induced metric is the unique metric for which the reference representative preserves the numerical inner product, namely \(g(V,W)=\bar g(PV,PW)\), or \(g=P^T\bar gP\).
\end{definition}

\subsection{Admissible deformation fields}\label{subsec:deformation_field}

The fixed-reference representation of tangent data has been fixed in \eqref{eq:coordinate_distinction_barV}.  We now specify the admissible fields that generate metric deformations.  The deformation field \(P\) is modeled on the non-rotational, dilation--shear sector of \(GL(n,\mathbb R)\) with respect to \(\bar g\).  Pointwise \(\bar g\)-orthogonal rotations, and boosts in pseudo-Riemannian signature, are not included in \(P\), since they preserve \(\bar g\) and therefore do not generate a new metric through \(g=P^T\bar gP\).

In the positive-definite case, \(P\) is taken to be a pure dilation--shear deformation field when
\begin{equation}\label{eq:P_self_adjoint}
 P^{T}\bar g=\bar g P,
 \qquad
 \bar g(PX,Y)=\bar g(X,PY),
\end{equation}
and \(P\) is positive-definite.  This restriction removes the reference-orthogonal sector from the metric-generating data.  A reference-orthogonal field \(O\) satisfies
\begin{equation}
 O^T\bar g O=\bar g,
\end{equation}
and therefore rotates or boosts reference frames without changing the metric.  Such transformations may occur as frame changes, but they are not counted as dilation--shear deformations.

\begin{proposition}[Polar-decomposition motivation]
Let $(E,\bar g)$ be a finite-dimensional positive-definite inner product space.  Let \(A\colon E\to E\) be invertible.  Then
\[
 A=OP,
\]
where $O$ is $\bar g$-orthogonal and $P$ is positive-definite and $\bar g$-self-adjoint.  Moreover,
\[
 A^T\bar g A=P^T\bar g P.
\]
\end{proposition}

\begin{proof}
Let $A^*$ denote the adjoint of $A$ with respect to $\bar g$.  The operator $A^*A$ is positive-definite and $\bar g$-self-adjoint, so it has a unique positive square root $P=(A^*A)^{1/2}$.  Put $O=AP^{-1}$.  Then
\[
 O^*O=(P^{-1})^*A^*AP^{-1}=P^{-1}P^2P^{-1}=I,
\]
so $O$ is $\bar g$-orthogonal.  Since $O^T\bar g O=\bar g$, one has
\[
 A^T\bar g A=P^TO^T\bar g OP=P^T\bar g P.
\]
Thus the metric-changing part of $A$ is precisely the positive self-adjoint factor $P$.
\end{proof}

\begin{remark}[Indefinite signature]
For pseudo-Riemannian metrics, polar decompositions require additional spectral and causal restrictions.  In indefinite signature the restricted dilation--shear sector should therefore be imposed as part of the geometric data rather than inferred from an arbitrary element of $GL(n,\mathbb R)$.
\end{remark}

\subsection{Metric induced by a deformation field}\label{subsec:deformed_metric}

Given a selected reference metric \(\bar g\) and an admissible deformation field \(P\), the induced metric is
\begin{equation}\label{eq:metric_components}
 g_{\mu\nu}=\bar g_{AB}P^{A}{}_{\mu}P^{B}{}_{\nu},
\end{equation}
or equivalently
\begin{equation}\label{eq:metric_law}
 g=P^{T}\bar g P .
\end{equation}
This subsection records the metric consequence of the reference representation principle; the comparison law associated with \(P\) is developed separately in Section~\ref{sec:connections}.

\begin{proposition}[Local existence and uniqueness of the deformation field]
Let $(M,\bar g)$ be a Riemannian manifold, and let $g$ be another Riemannian metric on the same manifold. Then, locally on $M$, there exists a unique $\bar g$-symmetric positive-definite $(1,1)$-tensor field $P$ such that
\[
 g=P^T\bar g\,P .
\]
\end{proposition}

\begin{proof}
Since both $\bar g$ and $g$ are positive-definite, at each point $x\in M$ there exists a unique $\bar g$-self-adjoint positive-definite endomorphism $A_x:T_xM\to T_xM$ such that
\[
 g_x(u,v)=\bar g_x(A_xu,v)
 \qquad\text{for all }u,v\in T_xM .
\]
Because $A_x$ is positive-definite and $\bar g$-self-adjoint, it admits a unique positive-definite $\bar g$-self-adjoint square root $P_x=A_x^{1/2}$.  Then
\[
 g_x(u,v)=\bar g_x(P_x^2u,v)=\bar g_x(P_xu,P_xv),
\]
which is precisely $g_x=P_x^T\bar g_xP_x$.  Smooth dependence of the positive-definite square root on the coefficients of a smooth positive-definite endomorphism gives a locally smooth tensor field $P$.

For uniqueness, suppose $P$ and $Q$ are both $\bar g$-symmetric positive-definite and satisfy $P^T\bar gP=Q^T\bar gQ$.  Then $P^2=Q^2$ as $\bar g$-self-adjoint positive-definite endomorphisms.  By uniqueness of the positive-definite square root, $P=Q$.
\end{proof}

\section{The Natural Connection}\label{sec:connections}

After the deformation field and induced metric have been defined, the next object is the comparison law.  The natural comparison of the dilation--shear geometry is obtained by comparing reference representatives through \(P\).  The induced metric also has a Levi-Civita connection, but that connection belongs to the metric layer.  It is the torsion-free metric connection of \((M,g)\), not an additional term to be added to the natural connection.

\begin{remark}[Connection symbols]\label{rem:connection_symbols}
The barred derivative \(\bar\nabla\) and barred connection coefficients \(\bar\Gamma\) refer to the selected reference comparison.  The unbarred symbol \(\nabla\) denotes the covariant derivative obtained on the deformed space by pulling that comparison back through \(P\).  Its local connection coefficients are denoted by \(\Gamma\).  The symbol \(\Lambda\) denotes the total natural dilation--shear compensation; in the natural comparison convention used here,
\[
 \Lambda=\Gamma .
\]
The circled connection \(\mathring\Gamma[g]\) is the Levi-Civita connection of the induced metric \(g\).  It is used when the metric layer alone is selected, or when a fully isometric realization is being described.
\end{remark}

\subsection{Pullback of the reference comparison}\label{subsec:DS_compensation}

Let \(\bar V=PV\).  The reference derivative compares \(\bar V\) by
\[
 \bar\nabla \bar V=d\bar V+\bar\Gamma\bar V .
\]
Pulling this comparison back by \(P^{-1}\) defines the deformed-space covariant derivative associated with the natural connection:
\begin{equation}\label{eq:natural_connection_main}
 \nabla V:=P^{-1}\bar\nabla(PV).
\end{equation}
In local connection-form notation,
\begin{align}\label{eq:pullback_total_connection}
 \nabla V
 &=P^{-1}\bigl(d(PV)+\bar\Gamma PV\bigr)\nonumber\\
 &=dV+\bigl(P^{-1}\bar\Gamma P+P^{-1}dP\bigr)V.
\end{align}
Consequently the local connection coefficients of \(\nabla\) are
\begin{equation}\label{eq:Gamma_general_reference}
 \Gamma=P^{-1}\bar\Gamma P+P^{-1}dP.
\end{equation}
The total natural dilation--shear compensation is represented by the natural connection coefficients:
\begin{equation}\label{eq:Lambda_equals_Gamma_general}
 \Lambda:=\Gamma.
\end{equation}
Only in a fixed flat Cartesian reference frame does \(\bar\Gamma=0\), so that
\begin{equation}\label{eq:Lambda_components_flat}
 \Gamma=\Lambda=P^{-1}dP,
 \qquad
 \Lambda^{\rho}{}_{\mu\nu}
 =(P^{-1})^{\rho}{}_{A}\,\partial_\mu P^{A}{}_{\nu}.
\end{equation}
If the selected reference already has a comparison law, including the Levi-Civita comparison of a fully isometrically realized Riemannian reference, then the transported term \(P^{-1}\bar\Gamma P\) is part of the natural connection on the deformed space.

\begin{proposition}[Pullback of reference comparison]\label{prop:pullback_reference_comparison}
Assume the reference representation principle \(\bar V=PV\).  Let the reference representatives be compared by \(\bar\nabla\).  Then the pulled-back rule
\[
 \nabla V=P^{-1}\bar\nabla(PV)
\]
is a covariant derivative on the deformed space, and its local connection coefficients are given by \eqref{eq:Gamma_general_reference}.  Equivalently, the natural parallel condition \(\bar\nabla(PV)=0\) is
\begin{equation}\label{eq:reference_zero_equiv_deformed_zero}
 dV+\Gamma V=0,
 \qquad
 \Gamma=\Lambda=P^{-1}\bar\Gamma P+P^{-1}dP.
\end{equation}
\end{proposition}

\begin{proof}
Using the product rule in a local frame,
\[
 \bar\nabla(PV)=d(PV)+\bar\Gamma PV
             =P\,dV+(dP+\bar\Gamma P)V.
\]
Multiplying by \(P^{-1}\) gives
\[
 P^{-1}\bar\nabla(PV)
 =dV+\left(P^{-1}\bar\Gamma P+P^{-1}dP\right)V.
\]
This proves both the connection coefficient formula and the parallel condition.
\end{proof}

\subsection{Relation with the Levi-Civita metric connection}\label{subsec:natural_total_comparison_principle}

The Levi-Civita connection \(\mathring\Gamma[g]\) is determined by the induced metric \(g=P^T\bar gP\) after one imposes
\[
 \mathring\nabla g=0,
 \qquad
 \mathring T=0.
\]
These are defining conditions of the metric layer.  They do not define the natural dilation--shear connection.  In a fully isometric realization, this distinction is geometrically explicit: \(\mathring\Gamma[g]\) is the tangential metric-preserving, torsion-free connection of the realized Riemannian geometry, while the natural connection coefficients are obtained by transporting the selected reference comparison through \(P\).

Thus the natural total connection is not \(\mathring\Gamma[g]+\Lambda\).  It is already
\[
 \Gamma=\Lambda=P^{-1}\bar\Gamma P+P^{-1}dP .
\]
If \(\Gamma\) is compared with \(\mathring\Gamma[g]\), their difference
\begin{equation}\label{eq:relative_distortion_only}
 D:=\Gamma-\mathring\Gamma[g]
\end{equation}
is a relative distortion tensor between two specified connections.  It is not the original dilation--shear compensation, and it should not be used to redefine the natural connection.

\begin{definition}[Natural connection]\label{def:total_connection_new_logic}
The covariant derivative of the dilation--shear deformation geometry is \(\nabla V=P^{-1}\bar\nabla(PV)\).  The associated natural connection is represented locally by the coefficients \(\Gamma\), where
\[
 \Gamma=P^{-1}\bar\Gamma P+P^{-1}dP,
\]
and the total natural dilation--shear compensation is \(\Lambda=\Gamma\).  The Levi-Civita connection coefficients \(\mathring\Gamma[g]\) are the connection coefficients of the metric layer, not an additional summand in \(\Gamma\).
\end{definition}

\subsection{Reference choice and composed deformation}\label{subsec:fixed_reference_composed_v2}

A deformation field may act on any selected metric-compatible reference geometry.  At the metric level,
\begin{equation}\label{eq:metric_level_arbitrary_reference_v3}
 g=P^T\bar gP.
\end{equation}
At the comparison level the same pullback rule applies:
\begin{equation}\label{eq:compensation_general_reference_form}
 \nabla_XV=P^{-1}\bar\nabla_X(PV),
\end{equation}
or locally,
\begin{equation}\label{eq:unified_natural_connection_formula}
 \Gamma=\Lambda=P^{-1}\bar\Gamma P+P^{-1}dP.
\end{equation}
Thus a flat Euclidean or Minkowskian baseline, a fully isometrically realized Riemannian or pseudo-Riemannian reference, an already deformed metric-compatible reference, and a composite reference carrying both embedded metric data and previous internal dilation--shear data are not governed by separate formulas.  They are different choices of \((\bar g,\bar\nabla,\bar\Gamma)\) in the same rule.

\begin{proposition}[Uniform pullback rule for composed deformations]\label{prop:unified_composition_rule}
Let \((M,\bar g,\bar\nabla)\) be a metric-compatible reference geometry, \(\bar\nabla\bar g=0\).  Let \(P_1\) and \(P_2\) be two successive admissible dilation--shear deformations, with
\begin{equation}\label{eq:g1_composed_unified}
 g_1=P_1^T\bar gP_1,
\end{equation}
\begin{equation}\label{eq:g2_composed_unified}
 g_2=P_2^Tg_1P_2.
\end{equation}
Let \(\Gamma_1\) be the natural connection coefficients of the first deformed space:
\begin{equation}\label{eq:Gamma1_composed_unified}
 \Gamma_1=P_1^{-1}\bar\Gamma P_1+P_1^{-1}dP_1.
\end{equation}
Then the natural connection coefficients of the second deformed space, using the first deformed space as reference, are
\begin{equation}\label{eq:Gamma2_stepwise_unified}
 \Gamma_2=P_2^{-1}\Gamma_1P_2+P_2^{-1}dP_2.
\end{equation}
Equivalently, relative to the original reference, the total deformation field is
\begin{equation}\label{eq:Ptot_composed_v2}
 P_{\rm tot}=P_1P_2,
\end{equation}
and
\begin{equation}\label{eq:Gamma_total_composed_unified}
 \Gamma_{\rm tot}=\Lambda_{\rm tot}
 =(P_1P_2)^{-1}\bar\Gamma(P_1P_2)+(P_1P_2)^{-1}d(P_1P_2).
\end{equation}
The two expressions agree:
\begin{equation}\label{eq:Gamma_total_expanded_unified}
 \Gamma_{\rm tot}=P_2^{-1}\Gamma_1P_2+P_2^{-1}dP_2.
\end{equation}
If the reference comparison is metric-compatible, then every stage remains metric-compatible with its induced metric.
\end{proposition}

\begin{proof}
The first step follows directly from the pullback rule applied to \(P_1\).  For the second step, the reference comparison is now represented by \(\Gamma_1\), so applying the same rule to \(P_2\) gives \eqref{eq:Gamma2_stepwise_unified}.  Relative to the original reference one has \(P_{\rm tot}=P_1P_2\).  Expanding \eqref{eq:Gamma_total_composed_unified} gives
\begin{align*}
\Gamma_{\rm tot}
&=(P_1P_2)^{-1}\bar\Gamma(P_1P_2)+(P_1P_2)^{-1}d(P_1P_2)\\
&=P_2^{-1}\bigl(P_1^{-1}\bar\Gamma P_1+P_1^{-1}dP_1\bigr)P_2+P_2^{-1}dP_2\\
&=P_2^{-1}\Gamma_1P_2+P_2^{-1}dP_2.
\end{align*}
Metric compatibility at each stage follows from Theorem~\ref{thm:natural_metric_compatibility}, because the reference metric and reference comparison at that stage are metric-compatible.
\end{proof}

\begin{remark}[Composite references with embedded and internal structures]\label{rem:composite_reference_interpretation}
The same formula also applies to the most general composite reference situation considered here.  The selected reference may be obtained by applying an internal dilation--shear deformation to a space that has already been fully isometrically realized, rather than by deforming a flat Euclidean or Minkowskian baseline directly.  Once this composite space is selected as the reference, its present metric-compatible comparison is denoted by \(\bar\Gamma\), and any further deformation \(P\) is governed by the same natural-connection formula
\[
 \Gamma=\Lambda=P^{-1}\bar\Gamma P+P^{-1}dP.
\]
No separate Levi-Civita term or embedding correction is added to this formula; the Levi-Civita connection is relevant only when one chooses to compare the natural connection with the metric layer.
\end{remark}

\section{Torsion, Nonmetricity, and Curvature}\label{sec:torsion_curvature}

This section records the basic tensorial structures associated with the natural connection.  The essential point is that a natural connection obtained from a metric-compatible reference is itself metric-compatible with the induced metric.  Its departure from the Levi-Civita metric connection is therefore generally expressed by torsion and by its deformation origin, not by nonmetricity.

\begin{theorem}[Metric compatibility of the natural connection]\label{thm:natural_metric_compatibility}
Let \((M,\bar g,\bar\nabla)\) be a reference metric geometry satisfying
\[
 \bar\nabla\bar g=0.
\]
Let \(P\) be an admissible dilation--shear deformation field, and define
\[
 g(U,V)=\bar g(PU,PV),
 \qquad\text{equivalently}\qquad
 g=P^T\bar gP.
\]
Define the deformed-space covariant derivative by
\[
 \nabla_XV=P^{-1}\bar\nabla_X(PV).
\]
Then \(\nabla\) is metric-compatible with the induced metric \(g\):
\[
 \nabla g=0.
\]
Consequently, the nonmetricity of the natural connection with respect to \(g\) vanishes.
\end{theorem}

\begin{proof}
For vector fields \(U,V\),
\[
 g(U,V)=\bar g(PU,PV).
\]
Using \(\bar\nabla\bar g=0\),
\begin{align*}
 X[g(U,V)]
 &=X[\bar g(PU,PV)]\\
 &=\bar g(\bar\nabla_X(PU),PV)
   +\bar g(PU,\bar\nabla_X(PV))\\
 &=\bar g(P\nabla_XU,PV)+\bar g(PU,P\nabla_XV)\\
 &=g(\nabla_XU,V)+g(U,\nabla_XV).
\end{align*}
This is precisely \(\nabla g=0\).  Hence the nonmetricity tensor
\[
 Q(X;U,V):=-(\nabla_Xg)(U,V)
\]
vanishes identically.
\end{proof}

\begin{remark}[Source of metric compatibility]
The reference spaces used in this paper ultimately come from Euclidean or Minkowskian metric spaces, or from Riemannian or pseudo-Riemannian metric geometries, together with metric-compatible comparisons.  A fully isometric realization also supplies a metric-compatible tangential comparison.  Therefore, when such a space is used as the reference for a dilation--shear deformation, the pullback construction preserves metric compatibility.  This explains why the deformation geometries considered here have zero nonmetricity with respect to their natural connection.  This statement does not identify the natural connection coefficients with the Levi-Civita connection coefficients of the induced metric: \(\mathring\Gamma[g]\) is the unique torsion-free metric-compatible connection determined by \(g\), whereas \(\Gamma=\Lambda\) is determined by the selected reference comparison and the deformation field \(P\), and may have torsion.
\end{remark}

\subsection{Torsion}\label{subsec:torsion}

The covariant derivative \(\nabla\) has local connection coefficients \(\Gamma^\rho{}_{\mu\nu}\).  In a coordinate frame its torsion is
\begin{equation}\label{eq:general_torsion_selected_connection}
 T^\rho{}_{\mu\nu}=\Gamma^\rho{}_{\mu\nu}-\Gamma^\rho{}_{\nu\mu}.
\end{equation}
In a fixed flat Cartesian baseline, \(\Gamma=\Lambda=P^{-1}dP\), hence
\begin{equation}\label{eq:torsion_Lambda_general}
 T^\rho{}_{\mu\nu}=2\Lambda^\rho{}_{[\mu\nu]}
 =(P^{-1})^\rho{}_A\left(\partial_\mu P^A{}_{\nu}-\partial_\nu P^A{}_{\mu}\right).
\end{equation}
Thus torsion measures the antisymmetric part of the natural dilation--shear comparison coefficients.  It may be nonzero even though the nonmetricity is zero.

\subsection{Nonmetricity}\label{subsec:nonmetricity}

For the natural connection, Theorem~\ref{thm:natural_metric_compatibility} gives
\begin{equation}\label{eq:natural_metric_compatible}
 \nabla g=0,
 \qquad
 Q_{\lambda\mu\nu}:=-(\nabla_\lambda g)_{\mu\nu}=0
\end{equation}
whenever \(\bar\nabla\bar g=0\).  Hence the dilation--shear deformation spaces considered here are not nonmetric geometries.  Their natural connection coefficients \(\Gamma\) may differ from \(\mathring\Gamma[g]\) because it may have torsion and because it is determined by the deformation field and the reference comparison.

\subsection{Curvature of the natural connection}\label{subsec:curvature_natural_connection}

The curvature two-form of the natural connection is
\begin{equation}\label{eq:curvature_general_coefficients}
 R[\Gamma]=d\Gamma+\Gamma\wedge\Gamma.
\end{equation}
In components,
\begin{equation}\label{eq:curvature_components_gamma}
 R^\rho{}_{\sigma\mu\nu}[\Gamma]
 =\partial_\mu\Gamma^\rho{}_{\nu\sigma}
 -\partial_\nu\Gamma^\rho{}_{\mu\sigma}
 +\Gamma^\rho{}_{\mu\lambda}\Gamma^\lambda{}_{\nu\sigma}
 -\Gamma^\rho{}_{\nu\lambda}\Gamma^\lambda{}_{\mu\sigma}.
\end{equation}
Because \(\Gamma=P^{-1}\bar\Gamma P+P^{-1}dP\), its curvature is the transported reference curvature:
\begin{equation}\label{eq:natural_R_reference}
 R[\Gamma]=P^{-1}\bar R P.
\end{equation}
Equivalently,
\begin{equation}\label{eq:natural_R_reference_components}
 R^\rho{}_{\sigma\mu\nu}[\Gamma]
 =(P^{-1})^\rho{}_A\,\bar R^A{}_{B\mu\nu}\,P^B{}_{\sigma}.
\end{equation}

\begin{proof}
Substitute \(\Gamma=P^{-1}\bar\Gamma P+P^{-1}dP\) into \(d\Gamma+\Gamma\wedge\Gamma\) and use \(d(P^{-1})=-P^{-1}(dP)P^{-1}\).  The mixed terms cancel, leaving
\[
 R[\Gamma]=P^{-1}(d\bar\Gamma+\bar\Gamma\wedge\bar\Gamma)P=P^{-1}\bar R P.
\]
\end{proof}

For a fixed flat Cartesian baseline, \(\bar\Gamma=0\) and \(\bar R=0\).  Thus
\begin{equation}\label{eq:natural_R_flat_zero}
 R[\Gamma]=0.
\end{equation}
In expanded form this is the Maurer--Cartan identity for \(\Gamma=P^{-1}dP\):
\begin{equation}\label{eq:maurer_cartan_flat_connection}
 d(P^{-1}dP)+(P^{-1}dP)\wedge(P^{-1}dP)=0.
\end{equation}
In components,
\begin{align}\label{eq:maurer_cartan_components}
0={}&\partial_\mu\!\bigl((P^{-1})^\rho{}_A\partial_\nu P^A{}_{\sigma}\bigr)
-\partial_\nu\!\bigl((P^{-1})^\rho{}_A\partial_\mu P^A{}_{\sigma}\bigr)\nonumber\\
&+\bigl((P^{-1})^\rho{}_A\partial_\mu P^A{}_{\lambda}\bigr)
  \bigl((P^{-1})^\lambda{}_B\partial_\nu P^B{}_{\sigma}\bigr)
-\bigl((P^{-1})^\rho{}_A\partial_\nu P^A{}_{\lambda}\bigr)
  \bigl((P^{-1})^\lambda{}_B\partial_\mu P^B{}_{\sigma}\bigr).
\end{align}
This identity shows that vanishing curvature of the flat-baseline natural connection does not mean that the deformation is uniform.  The non-uniformity may still be encoded in \(\Gamma=\Lambda\) and in torsion.

\subsection{Relation with Levi-Civita curvature}\label{subsec:lc_curvature_relation}

The Levi-Civita curvature \(R[\mathring\Gamma]\) is the curvature of the torsion-free metric connection of the induced metric \(g\).  It need not equal \(R[\Gamma]\).  Their difference is governed by the relative distortion tensor
\[
 D=\Gamma-\mathring\Gamma[g].
\]
For two connection one-forms related by \(\Gamma=\mathring\Gamma+D\), the corresponding curvatures satisfy
\begin{equation}\label{eq:curvature_difference_distortion}
 R[\Gamma]=R[\mathring\Gamma]+\mathring\nabla D+D\wedge D,
\end{equation}
where \(\mathring\nabla D=dD+\mathring\Gamma\wedge D+D\wedge\mathring\Gamma\) in connection-form notation.  Therefore, in a flat-baseline deformation, one has
\begin{equation}\label{eq:lc_curvature_from_distortion}
 R[\mathring\Gamma]=-\bigl(\mathring\nabla D+D\wedge D\bigr)
 \qquad\text{when }R[\Gamma]=0.
\end{equation}
This relation separates two kinds of curvature information: \(R[\Gamma]\) records the curvature transported from the selected reference comparison, while \(R[\mathring\Gamma]\) records the Riemannian curvature of the induced metric layer.

\subsection{Two decompositions of the natural connection}\label{subsec:lambda_split}

The natural connection coefficients \(\Gamma=\Lambda\) may be decomposed in two different ways.  The first is an internal decomposition of the natural comparison itself.  Lower the first vector index of \(\Lambda\) by the induced metric \(g\):
\begin{equation}
 \Lambda_{\mu\alpha\beta}=g_{\alpha\rho}\Lambda^\rho{}_{\mu\beta}.
\end{equation}
For each direction \(\mu\), split
\begin{equation}\label{eq:Lambda_split}
 \Lambda_{\mu\alpha\beta}
 =\widetilde\Lambda_{\mu\alpha\beta}+\Omega_{\mu\alpha\beta},
\end{equation}
where
\begin{equation}\label{eq:symmetric_Lambda}
 \widetilde\Lambda_{\mu\alpha\beta}:=\frac12(\Lambda_{\mu\alpha\beta}+\Lambda_{\mu\beta\alpha})
\end{equation}
and
\begin{equation}\label{eq:skew_Omega}
 \Omega_{\mu\alpha\beta}:=\frac12(\Lambda_{\mu\alpha\beta}-\Lambda_{\mu\beta\alpha}).
\end{equation}
In this internal split, \(\Omega\) is the natural rotational comparison compensation contained in the transport of vectors through the selected reference comparison and the field \(P\).  The term \(\widetilde\Lambda\) is the complementary deformation compensation associated with changes of lengths and mutual angles inside the same natural connection.  This split is not a nonmetricity split, because the covariant derivative associated with the covariant derivative associated with the natural connection is metric-compatible whenever the reference comparison is metric-compatible.  The rotational part \(\Omega\) is also not the Levi-Civita connection.  It is the natural rotation arising in the deformation comparison itself, while the Levi-Civita connection is the torsion-free metric-preserving rotational compensation selected by the induced metric layer.

The second decomposition compares the natural connection with the Levi-Civita connection of the induced metric layer.  In that relative description one writes
\begin{equation}\label{eq:LC_distortion_decomposition}
 \Gamma=\mathring\Gamma[g]+D,
 \qquad
 D:=\Gamma-\mathring\Gamma[g].
\end{equation}
Here \(\mathring\Gamma[g]\) is the Levi-Civita metric-preserving rotational compensation, and \(D\) is the residual deformation compensation measured relative to that Levi-Civita comparison.  Thus the internal split \(\Lambda=\widetilde\Lambda+\Omega\) and the relative split \(\Gamma=\mathring\Gamma+D\) answer different questions.  They are both legitimate descriptions, but they should not be combined into a formula of the form \(\mathring\Gamma+\Lambda\) for the natural total connection, because the natural total connection is already \(\Gamma=\Lambda\).

\section{Illustrative Families of Deformations}\label{sec:examples}

The following examples illustrate how the dilation--shear compensation and the Levi-Civita comparison of the induced metric differ in simple model families.  They exhibit the structures introduced above rather than repeating the ordinary Riemannian limiting case.  They also serve as consistency checks on the theory.  In each family one can test whether the formulas have the expected limiting behavior: constant deformations give no natural compensation over a flat Cartesian reference, one-dimensional dilation gives the same coefficient as the metric comparison, conformal dilation remains isotropic at the natural-connection level, non-diagonal shear can generate an internal rotational part, and the same metric gives different compensation data when different reference comparisons are selected.

\subsection{One-dimensional nonuniform dilation}

Let the reference line have metric \(d\bar s^2=dx^2\) and let
\begin{equation}
 P=a(x),\qquad a(x)>0 .
\end{equation}
Then
\begin{equation}
 g=a(x)^2dx^2 .
\end{equation}
The intrinsic dilation compensation is
\begin{equation}\label{eq:one_dim_Lambda}
 \Lambda_x=P^{-1}\partial_xP=\frac{a'}{a} .
\end{equation}
The same metric also has a Levi-Civita coefficient.  In this one-dimensional example the deformation compensation and the Levi-Civita coefficient coincide algebraically.  The Levi-Civita coefficient of the metric \(g=a^2dx^2\) is
\begin{equation}\label{eq:one_dim_LC}
 \mathring\Gamma^x{}_{xx}
 =\frac12g^{-1}\partial_xg
 =\frac{a'}{a} .
\end{equation}
Thus, numerically,
\begin{equation}\label{eq:one_dim_equality}
 \mathring\Gamma^x{}_{xx}=\Lambda^x{}_{xx} .
\end{equation}
This equality should not be read as an instruction to add the two terms.  It means that, in the one-dimensional dilation setting, the natural dilation--shear compensation and the Levi-Civita coefficient are two representations of the same dilation comparison: one is obtained directly from the deformation field \(P=a(x)\), and the other is obtained from the induced metric \(g=a^2dx^2\).  In a fully isometric realization, for instance after passing to the arclength coordinate
\begin{equation}
 s=\int a(x)\,dx,
\end{equation}
the same information is expressed by the Levi-Civita comparison of the Riemannian metric line.  This one-dimensional agreement should not be promoted to a general identification between \(\Gamma=\Lambda\) and \(\mathring\Gamma[g]\) in higher-dimensional deformation geometry.

The self-parallel equation in the intrinsic description is
\begin{equation}
 \ddot x+\frac{a'}{a}\dot x^2=0 .
\end{equation}
Equivalently, in the arclength coordinate, this is simply
\begin{equation}
 \frac{d^2s}{d\lambda^2}=0 .
\end{equation}
This example is important because it shows why \(\Lambda\) and \(\mathring\Gamma\) must not be added blindly.  In the one-dimensional dilation case they coincide as two representations of the same comparison information, whereas in higher dimensions they generally represent different comparison principles.  It is therefore a basic sanity check: the natural comparison reduces to the ordinary arclength comparison when there is only one possible tangent direction.

\subsection{A planar dilation--shear family}

Let the fixed reference geometry be the Euclidean plane with Cartesian coordinates $(x,y)$,
\begin{equation}
 \bar g=dx^2+dy^2,
 \qquad \bar{\Gamma}^k{}_{ij}=0 .
\end{equation}
Consider a diagonal deformation field
\begin{equation}\label{eq:diag_P_expanded}
 P(x,y)=\begin{pmatrix} a(x,y)&0\\0&b(x,y)\end{pmatrix},
 \qquad a>0,\quad b>0 .
\end{equation}
Then
\begin{equation}
 g=a^2dx^2+b^2dy^2 .
\end{equation}
The dilation--shear compensation is
\begin{equation}\label{eq:diag_Lambda_expanded}
 \Lambda_x=\begin{pmatrix} \partial_x\log a&0\\0&\partial_x\log b\end{pmatrix},
 \qquad
 \Lambda_y=\begin{pmatrix} \partial_y\log a&0\\0&\partial_y\log b\end{pmatrix} .
\end{equation}
If $a$ and $b$ are constant, this compensation vanishes.  If $a$ or $b$ varies, the flat plane has a nonzero dilation--shear deformation compensation even though the reference Levi-Civita connection in Cartesian coordinates is zero.

The Levi-Civita connection of the Riemannian metric $g=a^2dx^2+b^2dy^2$ is, in general, different from \eqref{eq:diag_Lambda_expanded}.  Its nonzero coefficients are
\begin{align}
 \oconn^x{}_{xx}&=\frac{a_x}{a},
 &\oconn^x{}_{xy}&=\frac{a_y}{a},
 &\oconn^x{}_{yy}&=-\frac{b b_x}{a^2},\\
 \oconn^y{}_{yy}&=\frac{b_y}{b},
 &\oconn^y{}_{xy}&=\frac{b_x}{b},
 &\oconn^y{}_{xx}&=-\frac{a a_y}{b^2}.
\end{align}
The first two diagonal-type coefficients resemble the deformation compensation, while the cross terms such as $\oconn^x{}_{yy}$ and $\oconn^y{}_{xx}$ arise from metric compatibility of the Riemannian metric geometry.  They are part of the Levi-Civita metric-preserving rotational comparison of the Riemannian metric geometry, not the direct dilation--shear deformation compensation.

For a vector $V=(V^x,V^y)^T$, the deformation comparison equations in the flat-reference planar case are
\begin{align}
 dV^x+(d\log a)V^x&=0,\\
 dV^y+(d\log b)V^y&=0.
\end{align}
Equivalently, $aV^x$ and $bV^y$ are constant along a deformation-parallel displacement.  This is exactly the fixed-reference statement that the reference representative $PV=(aV^x,bV^y)^T$ is unchanged.

This example separates two facts.  The metric $g$ records the completed intrinsic length structure.  The tensor $P^{-1}dP$ records the local variation of the dilation axes themselves.  The Levi-Civita connection of $g$ is metric-compatible and therefore reorganizes the same metric variation into a length-preserving comparison law.  This is a useful check on the formalism: when $a$ and $b$ are constant, the natural connection vanishes over the flat Cartesian reference, while the induced metric is only globally rescaled in fixed directions; when the transverse derivatives are present, the natural torsion and the Levi-Civita cross terms appear in different places, exactly as the theory predicts.

\subsection{Two-dimensional diagonal dilation--shear in detail}\label{subsec:diagonal_dilation_shear_detail_v2}

The planar family \eqref{eq:diag_P_expanded} is worth recording in a slightly more explicit form.  Put
\begin{equation}\label{eq:diag_exp_AB_v2}
 P=\begin{pmatrix}e^A&0\\0&e^B\end{pmatrix},
 \qquad
 g=\begin{pmatrix}e^{2A}&0\\0&e^{2B}\end{pmatrix},
\end{equation}
where \(A=A(x^1,x^2)\) and \(B=B(x^1,x^2)\).  Since \(P\) and \(\partial_\mu P\) commute in this diagonal example,
\begin{equation}\label{eq:diag_Lambda_AB_v2}
 \Lambda_1=\begin{pmatrix}\partial_1A&0\\0&\partial_1B\end{pmatrix},
 \qquad
 \Lambda_2=\begin{pmatrix}\partial_2A&0\\0&\partial_2B\end{pmatrix}.
\end{equation}
The deformation-induced rotational part vanishes:
\begin{equation}\label{eq:diag_Omega_zero_v2}
 \Omega_1=0,
 \qquad
 \Omega_2=0.
\end{equation}
Thus a diagonal dilation--shear field with fixed axes has no internal rotational compensation.  Its compensation is purely in the symmetric dilation--shear sector.

The torsion of the natural connection may nevertheless be nonzero, because torsion depends on the antisymmetry in the two lower connection indices, not on the $g$-skew part of each matrix $\Lambda_\mu$.  From \eqref{eq:diag_Lambda_AB_v2},
\begin{equation}\label{eq:diag_torsion_AB_v2}
 T^1{}_{12}=-\partial_2A,
 \qquad
 T^2{}_{12}=\partial_1B .
\end{equation}
Thus anisotropic dilation along fixed axes can induce torsion whenever the dilation in one direction varies transversely to that direction.  For the covariant derivative \(\nabla\) whose connection coefficients are the natural connection coefficients, however, the nonmetricity relative to the induced metric is zero by Theorem~\ref{thm:natural_metric_compatibility}.

The Levi-Civita connection \(\oconn[g]\) of the Riemannian metric is nevertheless not identical to \(\Lambda\).  Some representative coefficients are
\begin{align}\label{eq:diag_LC_coeffs_v2}
 \oconn^1{}_{11}&=\partial_1A,&
 \oconn^2{}_{12}&=\partial_1B,\\
 \oconn^1{}_{12}&=\partial_2A,&
 \oconn^2{}_{22}&=\partial_2B,\\
 \oconn^1{}_{22}&=-e^{2B-2A}\partial_1B,&
 \oconn^2{}_{11}&=-e^{2A-2B}\partial_2A.
\end{align}
The extra cross terms are the price of metric compatibility of the Riemannian metric.  They belong to the Levi-Civita comparison law, not to the direct dilation--shear compensation in that example.  This example gives a transparent calculation of the difference between deformation compensation and metric-compatible comparison even before non-diagonal shear is introduced.

It is also useful to compute explicitly the relative distortion
\begin{equation}\label{eq:diag_distortion_definition_v2}
 D:=\Gamma-\mathring\Gamma[g].
\end{equation}
Using \eqref{eq:diag_Lambda_AB_v2} and \eqref{eq:diag_LC_coeffs_v2}, the distortion matrices in the two coordinate directions are
\begin{equation}\label{eq:diag_distortion_matrices_v2}
 D_1=
 \begin{pmatrix}
 0&-\partial_2A\\[2pt]
 e^{2A-2B}\partial_2A&0
 \end{pmatrix},
 \qquad
 D_2=
 \begin{pmatrix}
 0&e^{2B-2A}\partial_1B\\[2pt]
 -\partial_1B&0
 \end{pmatrix}.
\end{equation}
Equivalently, the nonzero components are
\begin{equation}\label{eq:diag_distortion_components_v2}
 D^1{}_{12}=-\partial_2A,\qquad
 D^2{}_{11}=e^{2A-2B}\partial_2A,\qquad
 D^1{}_{22}=e^{2B-2A}\partial_1B,\qquad
 D^2{}_{21}=-\partial_1B .
\end{equation}
Thus the distortion vanishes precisely for separable one-directional dilation in this diagonal example, namely when \(\partial_2A=0\) and \(\partial_1B=0\).  These are also the transverse derivatives that produce the torsion components in \eqref{eq:diag_torsion_AB_v2}.  This makes the role of \(D\) transparent: it measures the part by which the natural comparison \(\Gamma=\Lambda\) differs from the torsion-free Levi-Civita comparison of the same induced metric.

The special case \(A=B=\varphi\) is another diagnostic limit.  Then \(P=e^\varphi I\) and the family reduces to the conformal deformation of Subsection~\ref{subsec:conformal_deformation_check}.  In this case \(\Gamma=d\varphi\,I\) is a pure scalar dilation connection, while the relative distortion becomes
\begin{equation}\label{eq:diag_distortion_conformal_limit_v2}
 D_1=
 \begin{pmatrix}
 0&-\partial_2\varphi\\[2pt]
 \partial_2\varphi&0
 \end{pmatrix},
 \qquad
 D_2=
 \begin{pmatrix}
 0&\partial_1\varphi\\[2pt]
 -\partial_1\varphi&0
 \end{pmatrix}.
\end{equation}
Thus, in the two-dimensional conformal limit, the natural connection is pure dilation, whereas its difference from the Levi-Civita comparison is purely rotational in the displayed matrix form.  This verifies that the Levi-Civita connection is not simply the dilation compensation itself; it reorganizes the same metric variation into the torsion-free metric comparison.

\subsection{Conformal deformation}\label{subsec:conformal_deformation_check}

A pure conformal deformation is obtained by
\begin{equation}
 P=e^{\varphi(x)}I .
\end{equation}
Then
\begin{equation}
 g=e^{2\varphi}\bar g,
\end{equation}
and
\begin{equation}\label{eq:conformal_DS}
 \Lambda_\mu=(\partial_\mu\varphi)I
\end{equation}
in a flat reference Cartesian frame, or $\Lambda_\mu=(\bar\nabla_\mu\varphi)I$ when the reference covariant derivative \(\bar\nabla\) is used.  In components over a flat reference chart,
\begin{equation}\label{eq:conformal_Lambda_components}
 \Lambda^k{}_{ij}=\delta^k_j\partial_i\varphi .
\end{equation}
If $\varphi$ is constant, $\Lambda=0$ and the deformation is a global homothety.  If $\varphi$ varies, the comparison detects a non-uniform local dilation.

This example is a useful test case for separating natural torsion from Levi-Civita metric curvature.  Since
\begin{equation}
 T^k{}_{ij}=\Lambda^k{}_{ij}-\Lambda^k{}_{ji},
\end{equation}
one obtains
\begin{equation}\label{eq:conformal_torsion}
 T^k{}_{ij}=\delta^k_j\partial_i\varphi-\delta^k_i\partial_j\varphi .
\end{equation}
Thus even an isotropic local dilation may produce a nonzero torsion of the natural connection when the dilation factor varies in different directions.  For the covariant derivative associated with this natural connection, \(Q(\nabla)=0\).  If the reference baseline is flat, then \(R[\Gamma]=0\).  These statements concern the natural connection coefficients \(\Gamma\), not the Levi-Civita connection \(\mathring\Gamma[g]\) of the conformally changed metric.

For a two-dimensional surface, the Gaussian curvature of a conformally related metric satisfies
\begin{equation}
 K[e^{2\varphi}\bar g]=e^{-2\varphi}\left(\bar K-\bar\Delta\varphi\right),
\end{equation}
where $\bar\Delta$ is the Laplace--Beltrami operator of $\bar g$ \cite[Thm.~7.30, p.~217]{Lee}.  Thus a nonconstant dilation field changes the intrinsic curvature of the Riemannian metric geometry.

This example tests the separation between the natural connection and the metric layer in its cleanest form.  Over a flat reference, \(R[\Gamma]=0\) by the Maurer--Cartan identity, while the Gaussian curvature of the induced metric may be nonzero when \(\bar\Delta\varphi\neq0\).  Hence zero curvature of the natural connection does not erase the effective Riemannian curvature of the induced metric; it relocates that curvature information into the distortion between \(\Gamma\) and \(\mathring\Gamma[g]\).

\subsection{From a fixed flat baseline patch to a spherical metric}\label{subsec:flat_to_sphere_v2}

This example separates three pieces of information: the metric generated by a deformation field, the dilation--shear compensation with respect to the chosen reference, and the ordinary Riemannian data of the resulting metric.

Take a flat reference metric written in local coordinates \((\theta,\phi)\) as
\begin{equation}
 \bar g=d\theta^2+d\phi^2 .
\end{equation}
Let
\begin{equation}\label{eq:P_flat_to_sphere_v2}
 P(\theta)=
 \begin{pmatrix}
 R&0\\[2pt]
 0&R\sin\theta
 \end{pmatrix} .
\end{equation}
Then
\begin{equation}\label{eq:sphere_metric_from_flat_v2}
 g=P^T\bar gP=R^2d\theta^2+R^2\sin^2\theta\,d\phi^2,
\end{equation}
which is the local metric of the round sphere of radius \(R\).  If this metric is merely represented from the flat patch by the field \(P\), the associated flat-patch representation compensation is
\begin{equation}\label{eq:Lambda_flat_to_sphere_v2}
 \Lambda_\theta=P^{-1}\partial_\theta P=
 \begin{pmatrix}
 0&0\\[2pt]
 0&\cot\theta
 \end{pmatrix},
 \qquad
 \Lambda_\phi=0 .
\end{equation}
Thus the same metric that is later recognized as a round sphere has a nonzero natural connection in this flat-patch representation.  This should not be confused with using the already realized round sphere as a reference, where its own Levi-Civita comparison is the reference comparison rather than an additional flat-to-sphere dilation--shear contribution.

The Levi-Civita connection of the metric \eqref{eq:sphere_metric_from_flat_v2} is the usual intrinsic Riemannian connection of the round sphere:
\begin{equation}\label{eq:sphere_LC_coeffs_v2}
 \mathring\Gamma^{\theta}{}_{\phi\phi}=-\sin\theta\cos\theta,
 \qquad
 \mathring\Gamma^{\phi}{}_{\theta\phi}=\mathring\Gamma^{\phi}{}_{\phi\theta}=\cot\theta,
\end{equation}
with the remaining independent coefficients equal to zero.  The Gaussian curvature is
\begin{equation}\label{eq:sphere_K_v2}
 K=\frac1{R^2} .
\end{equation}
Consequently a flat patch cannot be represented as a round spherical metric by a purely developable change of coordinates.  The metric representation requires a nontrivial \(P\), and \(\Lambda\) records that flat-patch representation.

The same intrinsic metric may now be locally realized as the ordinary round sphere in Euclidean three-space by
\begin{equation}\label{eq:round_sphere_embedding_chain_v2}
 X(\theta,\phi)=
 \bigl(R\sin\theta\cos\phi,\,R\sin\theta\sin\phi,\,R\cos\theta\bigr).
\end{equation}
Its induced metric is exactly \eqref{eq:sphere_metric_from_flat_v2}.  The unit normal is \(n=X/R\), and the second fundamental form is
\begin{equation}\label{eq:round_sphere_second_fundamental_v2}
 h_{\theta\theta}=R,\qquad
 h_{\phi\phi}=R\sin^2\theta,\qquad
 h_{\theta\phi}=0,
\end{equation}
up to the sign determined by the choice of normal.  Equivalently, \(h=(1/R)g\).  The Gauss equation gives
\begin{equation}
 K=\frac{h_{\theta\theta}h_{\phi\phi}-h_{\theta\phi}^2}
        {g_{\theta\theta}g_{\phi\phi}-g_{\theta\phi}^2}
  =\frac1{R^2},
\end{equation}
in agreement with \eqref{eq:sphere_K_v2}.  Thus two interpretations must be distinguished.  As a flat-patch representation, the spherical metric is written by a nontrivial field \(P\).  As a fully isometrically realized round sphere, the ordinary Riemannian description keeps \(g\), \(\mathring\Gamma[g]\), \(K\), and \(h\), and does not assign an independent dilation--shear compensation to the sphere with respect to itself.  This is an important reference-dependence check: the same metric may be produced from a flat patch by a nontrivial \(P\), but after full isometric realization the sphere is allowed to serve as its own Riemannian reference.  The formalism therefore avoids double-counting the flat-to-sphere representation when the realized sphere is used as the next reference.

\subsection{A genuinely non-diagonal shear family}

Let the fixed reference geometry again be the Euclidean plane and take a symmetric positive-definite deformation field
\begin{equation}\label{eq:P_nondiag_v2}
 P(x)=\begin{pmatrix} a(x)&s\\ s&1\end{pmatrix},
 \qquad 0<|s|<1,
 \qquad a(x)>s^2,
\end{equation}
where $s$ is a constant nonzero shear parameter.  The deformed metric is
\begin{equation}
 g=P^TP=P^2
 =\begin{pmatrix} a^2+s^2&s(a+1)\\ s(a+1)&1+s^2\end{pmatrix} .
\end{equation}
Since the reference connection is zero,
\begin{equation}
 \Lambda_1=P^{-1}\partial_1P,
 \qquad
 \Lambda_2=0 .
\end{equation}
Using
\begin{equation}
 P^{-1}=\frac{1}{a-s^2}\begin{pmatrix}1&-s\\-s&a\end{pmatrix},
\end{equation}
we obtain
\begin{equation}\label{eq:Lambda_nondiag_v2}
 \Lambda_1=\frac{a'}{a-s^2}\begin{pmatrix}1&0\\-s&0\end{pmatrix},
 \qquad
 \Lambda_2=0 .
\end{equation}
Although $P$ is symmetric, $P^{-1}\partial P$ need not be symmetric.  Thus a pure dilation--shear deformation field can induce a rotational component in the comparison law when the deformation varies spatially.

To display the decomposition invariantly, lower the first index with the deformed metric $g=P^2$:
\begin{equation}\label{eq:lowered_Lambda_nondiag_v2}
 \Lambda_{\rho 1\nu}:=g_{\rho\sigma}\Lambda^\sigma{}_{1\nu}
 =\begin{pmatrix}aa'&0\\sa'&0\end{pmatrix}.
\end{equation}
The $g$-symmetric and $g$-skew parts are therefore
\begin{equation}\label{eq:tLambda_nondiag_v2}
 \tLambda_{\rho1\nu}
 =\begin{pmatrix}aa'&\frac12sa'\\\frac12sa'&0\end{pmatrix},
\end{equation}
\begin{equation}\label{eq:Omega_nondiag_v2}
 \Omega_{\rho1\nu}
 =\begin{pmatrix}0&-\frac12sa'\\\frac12sa'&0\end{pmatrix}.
\end{equation}
Hence
\begin{equation}
 \Omega_1\neq0
 \qquad\text{whenever}\qquad sa'\neq0 .
\end{equation}
This is the essential point of the example.  The field $P$ in \eqref{eq:P_nondiag_v2} is symmetric and contains only dilation--shear as a metric-generating deformation.  Nevertheless, when the deformation varies spatially, the comparison of neighboring deformed vectors contains a rotational part.  The rotational sector $\Omega$ is therefore not an independent rotation inserted into $P$; it is induced by the spatial variation of a non-diagonal dilation--shear field.

For the one-variable choice $a=a(x)$ used above, the torsion components vanish because $\Lambda_2=0$ and the only nonzero column of $\Lambda_1$ is the first one.  This shows that the internal rotational part $\Omega$ and the torsion $T$ measure different aspects of the deformation compensation.  If the same family is allowed to depend on both coordinates, $a=a(x,y)$, then
\begin{equation}
 \Lambda_\mu=\frac{\partial_\mu a}{a-s^2}
 \begin{pmatrix}1&0\\-s&0\end{pmatrix},
\end{equation}
and the torsion has the representative components
\begin{equation}\label{eq:shear_torsion_twovar_v2}
 T^1{}_{12}=-\frac{\partial_2 a}{a-s^2},
 \qquad
 T^2{}_{12}=\frac{s\,
\partial_2 a}{a-s^2}.
\end{equation}
Thus non-diagonal shear can display both a $g$-skew internal rotational part and a torsion of the natural connection, but these are distinct derived structures.

For instance, if $a=1$, $a'=1$, and $s=1/2$, then
\begin{equation}
 \Lambda_1=\begin{pmatrix}4/3&0\\-2/3&0\end{pmatrix} .
\end{equation}
The nonzero off-diagonal term and the nonzero $g$-skew part show explicitly that non-diagonal shear can produce a rotational component inside the dilation--shear deformation compensation.  Omitting $\Omega$ would leave out part of the deformation compensation and would not give the full natural dilation--shear comparison.  This gives a nontrivial internal check on the decomposition of \(\Lambda\): a field \(P\) that is pointwise symmetric and contains no independent rotation can still generate a natural rotational comparison when its principal dilation--shear directions vary.  The rotational part \(\Omega\) and the torsion \(T\) are therefore not the same diagnostic; the example displays cases where one is present without the other.

\subsection{A non-uniform sphere over a round-sphere reference}\label{subsec:nonuniform_sphere_reference_v2}

A round sphere may itself be chosen as an intermediate metric reference for a further deformation.  Let
\begin{equation}\label{eq:round_sphere_reference_v2}
 \bar g_S=R^2d\theta^2+R^2\sin^2\theta\,d\phi^2,
 \qquad
 \bar K=\frac1{R^2},
\end{equation}
and let \(\bar\nabla\) be its Levi-Civita covariant derivative in this example, with connection coefficients \(\bar\Gamma_S\).  Since this reference is used as a fully isometrically realized Riemannian geometry, its comparison is its Levi-Civita comparison; it is not assigned an additional independent dilation--shear sector.

Now impose a non-uniform conformal dilation on the sphere:
\begin{equation}\label{eq:nonuniform_sphere_metric_v2_revised}
 g=e^{2\psi(\theta,\phi)}\bar g_S .
\end{equation}
This is represented by
\begin{equation}\label{eq:nonuniform_sphere_P_Lambda_v2_revised}
 P_S=e^{\psi}I .
\end{equation}
Because \(P_S\) is scalar, the natural total connection over the realized sphere is
\begin{equation}\label{eq:nonuniform_sphere_total_connection_v2_revised}
 \Gamma_S=\Lambda_S=P_S^{-1}\bar\Gamma_S P_S+P_S^{-1}dP_S
 =\bar\Gamma_S+d\psi\,I .
\end{equation}
The term \(\bar\Gamma_S\) is the transported tangential comparison of the realized round-sphere reference, while \(d\psi\,I\) is the additional dilation contribution.  Thus the new dilation contribution is nonzero whenever \(\psi\) is nonconstant.  This example describes a round sphere used as an intermediate metric reference and then made internally non-uniform by a new dilation field.  It is not the same as the ordinary round sphere itself.  The simplification \(P_S^{-1}\bar\Gamma_SP_S=\bar\Gamma_S\) is itself a useful check: it occurs because \(P_S=e^\psi I\) commutes with the reference connection.  For a general anisotropic dilation or shear over the sphere, the transported term \(P^{-1}\bar\Gamma_SP\) would be genuinely modified.

For the covariant derivative whose connection coefficients are given by the natural connection over the round-sphere reference, \(Q(\nabla)=0\).  The Gaussian curvature of the conformally non-uniform metric is
\begin{equation}\label{eq:nonuniform_sphere_K_v2_revised}
 K[g]=e^{-2\psi}\left(\frac1{R^2}-\Delta_S\psi\right),
\end{equation}
where \(\Delta_S\) is the Laplace--Beltrami operator of the round sphere.  Therefore the same scalar field that supplies the additional dilation contribution also changes the intrinsic Riemannian curvature of the non-uniform metric.

The corresponding natural comparison law is
\[
 \nabla_XV=P_S^{-1}\bar\nabla_X(P_SV).
\]
Here the intermediate round sphere supplies the reference covariant derivative \(\bar\nabla\) and its connection coefficients \(\bar\Gamma_S\).  The Levi-Civita connection \(\mathring\Gamma[g]\) of the non-uniform metric \(g\) is the corresponding metric-layer connection, not the natural comparison.

\subsection{Spherical compensations with respect to different references}\label{subsec:sphere_compensations_different_references_v2}

The previous two spherical examples should not be conflated.  A spherical metric can be written by a nontrivial field with respect to a flat patch, as in \eqref{eq:P_flat_to_sphere_v2}.  But once the same metric is selected as a fully isometrically realized round sphere, it is a metric-only Riemannian reference for the next comparison, and its dilation--shear compensation with respect to itself is zero.

If a further deformation is imposed on the round sphere, for example
\begin{equation}
 g=e^{2\chi(\theta,\varphi)}g_S,
\end{equation}
then there are two different questions.  With respect to the round sphere as a fully isometrically realized reference, the deformation field is \(P=e^\chi I\), and the natural total connection is
\begin{equation}\label{eq:sphere_reference_lambda_v2}
 \Gamma_{\mathrm{rel}}=\Lambda_{\mathrm{rel}}
 =P^{-1}\bar\Gamma_SP+P^{-1}dP
 =\bar\Gamma_S+d\chi\,I .
\end{equation}
Equivalently, in components,
\begin{equation}\label{eq:sphere_reference_lambda_components_v2}
 \Gamma^\rho{}_{\mu\nu}
 =\bar\Gamma^\rho{}_{S\mu\nu}+\delta^\rho_\nu\,\partial_\mu\chi .
\end{equation}
The term \(d\chi\,I\) is the new dilation contribution, while \(\bar\Gamma_S\) is the transported comparison of the realized round-sphere reference.  Formula \eqref{eq:sphere_reference_lambda_v2} is therefore correct as the natural connection formula.  It should not be confused with the Levi-Civita connection of the conformally changed metric \(e^{2\chi}g_S\), which contains the additional torsion-free symmetrizing terms required by the Levi-Civita condition.  This is a direct consistency test of the reference principle: the realized sphere contributes its own comparison once, and the further scalar deformation contributes only \(d\chi I\).  One should not reconstruct the round sphere from a flat patch and add a separate flat-to-sphere term.  Such a reconstruction would treat the already realized spherical bending as if it were still an independent dilation--shear deformation, which is not the interpretation adopted in this example.

\subsection{Summary of the illustrative families}

The examples illustrate and test three distinctions without repeating the same point.  First, in one dimension \(\Lambda\) and \(\mathring\Gamma\) may coincide algebraically as two representations of the same dilation comparison, but this is not a general identification.  Second, conformal, diagonal, and non-diagonal shear examples show that the deformation field may generate natural torsion and an internal rotational sector inside \(\Lambda\), while the covariant derivative associated with the natural connection remains metric-compatible with the induced metric.  The conformal limit \(A=B\) gives an especially useful check: the natural connection is pure scalar dilation, while the relative distortion from Levi-Civita is rotational in the displayed matrix form.  Third, the spherical examples show that a metric representation from a flat patch and a fully realized Riemannian reference are not the same comparison problem.  A round sphere may be represented from a flat patch by a nontrivial field, but once it is selected as a fully isometrically realized reference, its own Levi-Civita connection supplies the reference comparison rather than an additional flat-to-sphere dilation--shear term.  A further non-uniform deformation imposed on that sphere is then measured by the natural connection with respect to the realized round sphere.  These checks support the central rule: the natural connection is computed from the selected reference comparison and the actual deformation field being applied, not by adding independent representations of the same geometric information.

\section{Relation to Existing Geometric Frameworks}\label{sec:relations}

This section only locates the present construction relative to nearby frameworks.  It is not intended as a survey or a full comparative study.

\subsection{Riemannian geometry}

Riemannian geometry is an intrinsic theory of a metric geometry \((M,g)\) \cite{Lee,doCarmo}.  Its canonical comparison law is the Levi-Civita connection \(\mathring\Gamma[g]\), determined by
\begin{equation}
 \mathring\nabla g=0,\qquad \mathring T=0 .
\end{equation}
The present theory keeps this Levi-Civita comparison as the metric-layer connection, but also retains deformation data \(g=P^T\bar gP\) and the natural covariant derivative \(\nabla\), whose connection coefficients \(\Gamma\) are computed from the appropriate deformation field.  Thus the ordinary Riemannian case is recovered when no independent deformation comparison is retained, while the deformation-geometric case separates metric comparison from natural comparison.

\subsection{Embedding geometry}

An isometric embedding represents an intrinsic metric geometry by an ambient shape; see, for example, Nash's embedding theorem and standard moving-frame treatments \cite{Nash1956,SpivakVol2,KobayashiNomizu1}.  In an adapted moving frame, the ambient variation splits into a tangential Levi-Civita part and a normal second fundamental form.  Thus an embedding realizes the intrinsically defined Levi-Civita connection tangentially and records extrinsic bending through the normal part.  The dilation--shear compensation \(\Lambda\), however, is not defined by the embedding itself; it records the dilation--shear state computed from the relevant deformation field.  For an intrinsic final space this means the total deformation field from the fixed flat baseline, while a further deformation over a fully realized reference gives the compensation with respect to that realized reference.  In special examples a deformation field that produces a metric may be reflected, after isometric realization, in \(g\), \(\mathring\Gamma[g]\), curvature, and the second fundamental form, but \(\Lambda\) is not identical with extrinsic curvature.  A more general reference may already combine such an isometric realization with an internal dilation--shear natural connection.  In that case the reference is represented by its current metric-compatible connection \(\bar\Gamma\), and any further deformation is again computed from \(P^{-1}\bar\Gamma P+P^{-1}dP\).

\subsection{Metric-affine and symmetric teleparallel geometry}

Metric-affine geometry allows the metric and affine connection to be specified independently \cite{HehlEtAl}, and symmetric teleparallel geometry emphasizes nonmetricity in a flat, torsion-free affine representation \cite{NesterYo,BeltranJimenezEtAl}.  The present framework is different from both.  Here the covariant derivative and its connection coefficients are fixed together by the deformation field and the selected reference comparison:
\begin{equation}
 \nabla_XV=P^{-1}\bar\nabla_X(PV),
 \qquad
 \Gamma=\Lambda=P^{-1}\bar\Gamma P+P^{-1}dP.
\end{equation}
For a flat-baseline dilation--shear deformation the connection coefficients have torsion determined by \(\Lambda\), metric compatibility with the induced metric, and zero curvature when the reference connection is flat:
\begin{equation}
 T^\rho{}_{\mu\nu}(\nabla)=2\Lambda^\rho{}_{[\mu\nu]},
 \qquad
 Q(\nabla)=0,
 \qquad
 R[\Gamma]=0 .
\end{equation}
Thus the present construction is not a general nonmetric metric-affine geometry.  It is a metric-compatible deformation geometry with a connection determined by \(P\) and the reference comparison, rather than an independently postulated affine connection.

\subsection{Conformal geometry}

The conformal sector, related to the classical Weyl rescaling viewpoint \cite{Weyl1918}, is recovered when
\begin{equation}
 P=e^\varphi I .
\end{equation}
Then
\begin{equation}
 g=e^{2\varphi}\bar g,
 \qquad
 \Lambda=d\varphi\, I
\end{equation}
for a flat Cartesian reference frame.  Thus conformal geometry corresponds to the scalar dilation sector of the present framework, while the general theory also includes anisotropic dilation--shear.

\section{Discussion and Outlook}

This paper has developed the geometry of dilation- and shear-deformed spaces from a deformation-field viewpoint.  The metric is represented by a deformation field \(P\) through
\begin{equation}
 g=P^T\bar gP .
\end{equation}
The deformation field is not merely a square root of the metric.  It also supplies the reference representative \(\bar V=PV\) and therefore determines a natural comparison law.

The central point is the status of the total connection.  In the natural layer, the total comparison is encoded by the covariant derivative
\begin{equation}
 \nabla_XV=P^{-1}\bar\nabla_X(PV),
\end{equation}
with local connection coefficients
\begin{equation}
 \Gamma=P^{-1}\bar\Gamma P+P^{-1}dP.
\end{equation}
The total dilation--shear compensation is denoted by \(\Lambda\), and in this natural-comparison convention \(\Lambda=\Gamma\).  For an intrinsic final space relative to a fixed flat Cartesian baseline this becomes
\begin{equation}
 \Gamma_{\rm tot}=\Lambda_{\rm tot}=P_{\rm tot}^{-1}dP_{\rm tot}.
\end{equation}
This is the precise sense in which the total dilation--shear connection coefficients are the total deformation compensation.

A key result is that metric compatibility is preserved by the pullback construction.  If \(\bar\nabla\bar g=0\), then the induced metric \(g=P^T\bar gP\) satisfies
\[
 \nabla g=0.
\]
Thus the natural spaces considered here have zero nonmetricity.  Their general difference from Levi-Civita geometry is torsion and deformation dependence, not nonmetricity.

The Levi-Civita covariant derivative \(\mathring\nabla[g]\) and its connection coefficients \(\mathring\Gamma[g]\) remain essential, but their role is different.  The coefficients \(\mathring\Gamma[g]\) give the canonical metric-preserving, torsion-free connection of the induced metric layer.  They are not the natural comparison structure of the deformation field.  The distinction becomes especially transparent in a fully isometric realization, where the Levi-Civita connection appears as the tangential metric-preserving comparison of the realized Riemannian geometry, while any further dilation--shear compensation is computed from the further deformation field relative to the realized reference connection coefficients.

When one wants to compare the natural connection with the metric layer, one may use the relative distortion decomposition at the level of connection coefficients:
\begin{equation}
 \Gamma=\mathring\Gamma[g]+D,
 \qquad D=\Gamma-\mathring\Gamma[g].
\end{equation}
The expression \(\mathring\Gamma+\Lambda\) is therefore not used as the definition of the total natural connection.  The natural total connection is represented by \(\Gamma=\Lambda\); \(\mathring\Gamma\) belongs to the metric layer.

Embedded realizations provide a useful interpretation but not a substitute for the deformation-comparison data.  An isometric embedding preserves the metric \(g\) and realizes the intrinsic Levi-Civita connection tangentially, while its second fundamental form records normal bending.  If an already realized Riemannian geometry is used as a reference for a further deformation, then
\[
 P^{-1}\bar\Gamma P+P^{-1}dP
\]
is the natural total connection with respect to the realized reference; it is not a reconstruction of the reference's own bending.  If the same reference also carries an internal dilation--shear natural connection from an earlier deformation, that internal comparison is already included in \(\bar\Gamma\), and the same formula gives the final natural connection.  Equivalently, all earlier embedded and internal stages may be absorbed into a single total deformation field and a single total pullback formula.

The examples show how the theory behaves in simple cases.  In one dimension, the Levi-Civita and dilation expressions may coincide algebraically as two representations of the same dilation comparison.  In two and higher dimensions, conformal, anisotropic, and shear examples show that \(\mathring\Gamma\) and \(\Lambda\) encode different comparison principles.  The non-diagonal shear example also shows that \(\Lambda\) may contain a rotational part induced by the variation of the deformation field itself.

Future work should classify dilation--shear deformation fields by their natural connection coefficients, by their distortion relative to the Levi-Civita metric layer, and by their possible embedded realizations.  The same distinction should also be useful in applications where an effective metric is induced by a more fundamental deformation field.


\begin{thebibliography}{99}

\bibitem{BeltranJimenezEtAl}
J. Beltr\'an Jim\'enez, L. Heisenberg, and T. S. Koivisto,
The geometrical trinity of gravity,
\textit{Universe} 5 (2019), no. 7, 173.
\newblock \href{https://doi.org/10.3390/universe5070173}{doi:10.3390/universe5070173}.

\bibitem{doCarmo}
M. P. do Carmo,
\textit{Riemannian Geometry},
Mathematics: Theory \& Applications, Birkh\"auser, Boston, 1992.

\bibitem{HehlEtAl}
F. W. Hehl, J. D. McCrea, E. W. Mielke, and Y. Ne'eman,
Metric-affine gauge theory of gravity: Field equations, Noether identities, world spinors, and breaking of dilation invariance,
\textit{Phys. Rep.} 258 (1995), no. 1--2, 1--171.
\newblock \href{https://doi.org/10.1016/0370-1573(94)00111-F}{doi:10.1016/0370-1573(94)00111-F}.

\bibitem{KobayashiNomizu1}
S. Kobayashi and K. Nomizu,
\textit{Foundations of Differential Geometry. Vol. I},
Interscience Publishers, New York, 1963.

\bibitem{Lee}
J. M. Lee,
\textit{Introduction to Riemannian Manifolds}, 2nd ed.,
Graduate Texts in Mathematics 176, Springer, Cham, 2018.
\newblock \href{https://doi.org/10.1007/978-3-319-91755-9}{doi:10.1007/978-3-319-91755-9}.

\bibitem{Nash1956}
J. Nash,
The imbedding problem for Riemannian manifolds,
\textit{Ann. of Math. (2)} 63 (1956), no. 1, 20--63.
\newblock \href{https://doi.org/10.2307/1969989}{doi:10.2307/1969989}.

\bibitem{NesterYo}
J. M. Nester and H.-J. Yo,
Symmetric teleparallel general relativity,
\textit{Chinese Journal of Physics} 37 (1999), 113.
\newblock arXiv:\href{https://arxiv.org/abs/gr-qc/9809049}{gr-qc/9809049}.

\bibitem{SpivakVol2}
M. Spivak,
\textit{A Comprehensive Introduction to Differential Geometry. Vol. II},
2nd ed., Publish or Perish, Wilmington, 1979.

\bibitem{Weyl1918}
H. Weyl,
Gravitation und Elektrizit\"at,
\textit{Sitzungsberichte der K\"oniglich Preu{\ss}ischen Akademie der Wissenschaften zu Berlin} (1918), 465--480.

\end{thebibliography}
\end{document}